\documentclass{amsart}
\usepackage{amssymb}
\usepackage{amsmath}
\usepackage{amsfonts}
\usepackage{mathrsfs}

\newtheorem{theorem}{Theorem}[section]
\newtheorem{lemma}[theorem]{Lemma}

\theoremstyle{definition}

\theoremstyle{remark}
\newtheorem{remark}[theorem]{Remark}

\numberwithin{equation}{section}

\begin{document}

\title[BSE and BED properties of $L^1(G,\omega)$]{On the BSE and BED properties of the Beurling algebra $L^1(G,\omega)$}


\author[J. J. Dabhi]{Jekwin J. Dabhi}
\address{Institute of Infrastructure Technology Research and Management (IITRAM), Ahmedabad - 380026, Gujarat, India}
\curraddr{}
\email{jekwin.13@gmail.com}
\thanks{}

\author[P. A. Dabhi]{Prakash A. Dabhi}
\address{Institute of Infrastructure Technology Research and Management (IITRAM), Ahmedabad - 380026, Gujarat, India}
\curraddr{}
\email{lightatinfinite@gmail.com, prakashdabhi@iitram.ac.in}
\thanks{The first author is thankful  for  research support from CSIR HRDG, India (file no. 09/1274(13914)/2022-EMR-I) for the Senior Research Fellowship}



\date{}

\dedicatory{}

\begin{abstract}
Let $G$ be a locally compact abelian group, and let $\omega:G \to [1,\infty)$ be a weight, i.e.,  $\omega$ is measurable, $\omega$ is locally bounded and $\omega(s+t)\leq \omega(s)\omega(t)$ for all $s, t \in G$. If $\omega^{-1}$ is vanishing at infinity, then we show that the Beurling algebra $L^1(G,\omega)$ is both BSE- algebra and BED- algebra.
\end{abstract}

\maketitle
\section{Introduction}
Let $\mathcal A$ be a commutative Banach algebra, and let $\mathcal A$ be \emph{without order}, i.e., if $a \in \mathcal A$ and $a\mathcal A=\{0\}$, then $a=0$. A nonzero linear map $\varphi:\mathcal A \to \mathbb C$ is a \emph{complex homomorphism} if $\varphi(ab)=\varphi(a)\varphi(b)$ for all $a, b \in \mathcal A$. Let $\Phi_{\mathcal A}$ be the collection of all complex homomorphisms on $\mathcal A$. Then $\Phi_{\mathcal A}$ is a locally compact space with the weak$^\ast$- topology induced from the  weak topology on the dual $\mathcal A^\ast$ of $\mathcal A$. The space $\Phi_{\mathcal A}$ with this topology is the \emph{Gelfand space} of $\mathcal A$. For $a \in \mathcal A$, let $\widehat a:\Phi_{\mathcal A}\to \mathbb C$ be $\widehat a(\varphi)=\varphi(a)$ for all $\varphi \in \Phi_{\mathcal A}$. Then $\widehat a$ is the \emph{Gelfand transform} of $a$. Let $\widehat{A} = \left\lbrace \hat{a} : a \in A \right\rbrace $. By \cite[Theorem 2.2.7(i)]{E}, $\widehat a \in C_0(\Phi_{\mathcal A})$, the collection of all complex valued continuous function on $\Phi_{\mathcal A}$ vanishing at infinity. A commutative Banach algebra $\mathcal A$ is \emph{semisimple} if the map $a\mapsto \widehat a$ from $\mathcal A$ to $C_0(\Phi_{\mathcal A})$ is injective.

Let $\mathcal A$ be a commutative Banach algebra, and let $\mathcal A$ be without order. Clearly, if $\mathcal A$ is semisimple, then it is without order. A map $T:\mathcal A \to \mathcal A$ is a \emph{multiplier} if $T(ab)=aTb=(Ta)b$ for all $a, b \in \mathcal A$. Let $M(\mathcal A)$ be the collection of all multipliers on $\mathcal A$. Then $M(\mathcal A)$ is a closed subalgebra of $B(\mathcal A)$, the Banach algebra of all bounded linear maps from $\mathcal A$ to $\mathcal A$. The Banach algebra $M(\mathcal A)$ contains the identity function $I(a)=a\;(a\in \mathcal A)$ and $\mathcal A$ is an ideal in $M(\mathcal A)$ via the identification $a \mapsto T_a$, where $T_a(b)=ab$ for all $b \in \mathcal A$. If $T \in M(\mathcal A)$, then, by \cite[Theorem 1.2.2]{L}, there is a unique bounded continuous function $\sigma:\Phi_{\mathcal A}\to \mathbb C$ such that $\widehat{(Ta)}(\varphi)=\sigma(\varphi)\widehat a(\varphi)$ for all $\varphi \in \Phi_{\mathcal A}$ and $a \in \mathcal A$. We denote this $\sigma$ by $\widehat T$. Thus $\widehat{(Ta)}(\varphi)=\widehat T(\varphi)\widehat a(\varphi)$ for all $T \in M(\mathcal A)$, $\varphi \in \Phi_{\mathcal A}$ and $a \in \mathcal A$. Let $\widehat {M(\mathcal A)}=\widehat M(\mathcal A)=\{\widehat T:T \in M(\mathcal A)\}$.

Let $G$ be a locally compact abelian group with the Haar measure denoted by $ds$ or $dm$, and let the binary operation on $G$ be denoted by $+$. Let $$L^1(G)=\left\{f:G \to \mathbb C:f\text{ is measurable, }\|f\|_1=\int_G |f(s)|ds<\infty\right\}.$$ Then $L^1(G)$ is a Banach space with the norm $\|\cdot\|_1$. For $f, g \in L^1(G)$, let
\begin{equation}\label{convolution}
(f\star g)(s) = \int_G f(t) g(s-t) dt \quad(s \in G).
\end{equation}

Then $f\star g$ is the \emph{convolution} of $f$ and $g$, $f\star g \in L^1(G)$ and $\|f\star g\|_1 \leq \|f\|_1 \|g\|_1$. The Banach space $L^1(G)$ is, in fact, a commutative Banach algebra with convolution as a multiplication. A continuous map $\gamma:G \to \mathbb T$ is a \emph{character} on $G$ if $\gamma(s+t)=\gamma(s)\gamma(t)$ for all $s,t \in G$. Let $\widehat G$ be the collection of all characters on $G$. Then $\widehat G$ is a group with the binary operation $(\gamma_1+\gamma_2)(s)=\gamma_1(s)\gamma_2(s)$ for all $\gamma_1, \gamma_2 \in \widehat G$ and $s \in G$. If $\gamma$ is a character on $G$, then $\varphi_\gamma(f)=\int_G f(s)\overline{\gamma(s)}ds\;(f \in L^1(G))$ is a complex homomorphism on $L^1(G)$. By \cite[Theorem 2.7.2]{E}, if $\varphi$ is a complex homomorphism on $L^1(G)$, then there is a unique $\gamma \in \widehat G$ such that $\varphi=\varphi_\gamma$. Thus $\widehat G$ can be identified with $\Phi_{L^1(G)}$ via the map $\gamma \mapsto \varphi_\gamma$. The group $\widehat G$ with the Gelfand topology is a locally compact abelian group and is referred as the dual group of $G$. Let $d\gamma$ denote the Haar measure on $\widehat G$. The Haar measure of the dual group $\widehat G$ of $G$, is normalized so that, for a specific class of functions, the inversion formula holds. If $f \in L^1(G)$, then its \emph{Fourier transform} (Gelfand transform) is the function $\widehat f:\widehat G \to \mathbb C$ is given by $\widehat{f}(\varphi_{\gamma}) =  \widehat f(\gamma)=\int_G f(s)\overline{\gamma(s)}ds$ for all $\gamma \in \widehat G$ and $\widehat f \in C_0(\widehat G)=C_0(\Phi_{L^1(G)})$. Note that $L^1(G)$ is semisimple \cite[Corollary 2.7.9]{E}.

Let $M(G)$ be the collection of all complex regular Borel measures on $G$. Then $M(G)$ is a commutative Banach algebra with convolution of measures as multiplication and the total variation norm. The Gelfand space $\Phi_{M(G)}$ of $M(G)$ contains $\widehat G$. If $\mu \in M(G)$, then its \emph{Fourier - Stieltjes transform} (Gelfand transform) $\widehat \mu$, restricted to $\widehat G$, is given by $\widehat{\mu}(\varphi_{\gamma}) = \widehat \mu(\gamma)=\int_G \overline{\gamma(s)}d\mu(s)$ for all $\gamma \in \widehat G$. If $\mu \in M(G)$ and if $\widehat \mu(\gamma)=0$ for all $\gamma \in \widehat G$, then $\mu=0$, \cite[$\S$ 1.7.3(b)]{R}. If $\mu \in M(G)$, then $T_\mu$ defined by $T_\mu(f)=\mu \star f\;(f \in L^1(G))$ is a multiplier on $L^1(G)$, where $(\mu\star f)(s)=\int_G f(s-t)d\mu(t)$ for all $s \in G$. Also, $\|T_\mu\|=\|\mu\|$. Conversely, by \cite[Theorem 0.1.1 ($\romannumeral 5$)]{L}, if $T$ is a multiplier on $L^1(G)$, then there is a unique $\mu \in M(G)$ such that $T=T_\mu$. Thus $M(G)$ is isometrically isomorphic to the multiplier algebra $M(L^1(G))$ of $L^1(G)$. The space $L^1(G)$ is an ideal in $M(G)$ via the identification $f \mapsto T_f$, where $T_f(g)=f\star g\;(g \in L^1(G))$.

Let $X$ be a locally compact Hausdorff space, and let $C_0(X)$ be the collection of all complex valued continuous functions on $X$ which are vanishing at infinity. Then $C_0(X)$ is a Banach algebra with pointwise operations and the sup norm $\|f\|_\infty=\sup\{|f(x)|:x \in X\}$ for all $f \in C_0(X)$. By Riesz representation theorem the dual $C_0(X)^\ast$ of $C_0(X)$ can be identified with the Banach space $M(X)$ of all complex regular Borel measures on $X$ via the identification $\mu \mapsto L_\mu$, $M(X)\to C_0(X)^\ast$, where $L_\mu(f)=\int_X f(x)d\mu(x)$ for all $f \in C_0(X)$.

Let $G$ be a locally compact abelian group. A map $\omega:G \to [1,\infty)$ is a \emph{weight} \cite[Definition 1.6.1]{H} if $\omega$ is measurable, $\omega$ is locally bounded, i.e., $\omega$ is bounded on every compact subset of $G$, and $\omega(s+t)\leq \omega(s)\omega(t)$ for all $s,t \in G$. Let $$L^1(G,\omega)=\left\{f:G \to \mathbb C: f \text{ is measurable, }\|f\|_\omega=\int_G |f(s)|\omega(s)ds<\infty\right\}.$$ Then $L^1(G,\omega)$ is a commutative Banach algebra with the convolution multiplication and the norm $\|\cdot\|_\omega$. A continuous map $\gamma:G \to \mathbb C\setminus\{0\}$ is an \emph{$\omega$- bounded generalized character} if $\gamma(s+t)=\gamma(s)\gamma(t)$ for all $s, t \in G$ and $|\gamma(s)|\leq \omega(s)$ for all $s \in G$. Let $\widehat G_\omega$ be the collection of all $\omega$- bounded generalized characters on $G$. If $\gamma \in \widehat G_\omega$, then $\varphi_\gamma$ is a complex homomorphism on $L^1(G,\omega)$, where $\varphi_\gamma(f)=\int_G f(s)\overline{\gamma(s)}ds$ for all $f \in L^1(G,\omega)$. Conversely, by \cite[Theorem 2.8.2.]{E}, if $\varphi$ is a complex homomorphism on $L^1(G,\omega)$, then there is a unique $\gamma \in \widehat G_\omega$ such that $\varphi=\varphi_\gamma$. This establishes an one to one correspondence between $\widehat G_\omega$ and $\Phi_{L^1(G,\omega)}$. We shall consider $\widehat G_\omega$ with the Gelfand topology on it. If $f \in L^1(G,\omega)$, then, as in the case of $L^1(G)$, the \emph{Fourier transform} $\widehat f:\widehat G_\omega \to \mathbb C$ is defined as $\widehat{f}(\varphi_{\gamma}) = \widehat f(\gamma)=\int_G f(s)\overline{\gamma(s)}ds$ for all $\gamma \in \widehat G_\omega$. Since $\omega\geq 1$, we have $L^1(G,\omega)\subseteq L^1(G)$ and $\widehat G \subseteq \widehat G_\omega$. Let $f \in L^1(G,\omega)$. Since $L^1(G,\omega)\subset L^1(G)$, we have $\widehat f_{|_{\widehat G}}\in C_0(\widehat G)$.

Let $M(G,\omega)$ be the collection of all regular complex Borel measure $\mu$ on $G$ such that $\|\mu\|_\omega=\int_G \omega(s)d|\mu|(s)<\infty$. Then $M(G,\omega)$ is a commutative Banach algebra with the convolution of measures and the norm $\|\cdot\|_\omega$. Note that $\widehat G_\omega\subset \Phi_{M(G,\omega)}$. If $\mu \in M(G,\omega)$, then the \emph{Gelfand transform} $\widehat \mu$ of $\mu$ restricted to $\widehat G_\omega$ is given by $\widehat{\mu}(\varphi_{\gamma}) = \widehat \mu(\gamma)=\int_G \overline{\gamma(s)}d\mu(s)$ for all $\gamma \in \widehat G_\omega$. If $\mu \in M(G,\omega)$, then $T_\mu$ is a multiplier on $L^1(G,\omega)$, where $T_\mu(f)=\mu\star f$ for all $f \in L^1(G,\omega)$ and conversely, by \cite[Lemma 2.3]{G}, if $T$ is a multiplier on $L^1(G,\omega)$, then there is a unique $\mu \in M(G,\omega)$ such that $T=T_\mu$. Also, $\|T_\mu\|=\|\mu\|_\omega$ for all $\mu \in M(G,\omega)$. Thus the multiplier algebra $M(L^1(G,\omega))$ of $L^1(G,\omega)$ can be identified with $M(G,\omega)$.

Let $\omega$ be a weight on a locally compact abelian group $G$, and let $C_0(G,\omega^{-1})$ be the collection of all complex valued continuous functions $f$ on $G$ such that $f\omega^{-1}$ is vanishing at infinity. Then $C_0(G,\omega^{-1})$ is a Banach space with pointwise operations and the norm $$\|f\|_{0,\omega^{-1}}=\sup\left\{\frac{|f(s)|}{\omega(s)}:s \in G\right\}\quad(f \in C_0(G,\omega^{-1})).$$ Using Riesz representation theorem, we have $C_0(G,\omega^{-1})^\ast$ is isometrically isomorphic to $M(G,\omega)$. An isometric isomorphism from $M(G,\omega)$ to $C_0(G.\omega^{-1})^\ast$ may be given by $\mu \mapsto T_\mu$, where $T_\mu(f)=\int_G f(s)d\mu(s)$ for all $f \in C_0(G,\omega^{-1})$.

Let $\mathcal A$ be a commutative Banach algebra, and let $\mathcal A$ be without order. A bounded continuous function $\sigma:\Phi_{\mathcal A} \to \mathbb C$ is a \emph{BSE- function} \cite{T} if there exists a positive constant $C$ such that $$\left|\sum_{i=1}^n c_i \sigma(\varphi_i)\right|\leq C\left\|\sum_{i=1}^n c_i\varphi_i\right\|_{\mathcal A^\ast}$$ for all $c_1, \ldots, c_n \in \mathbb C$, $\varphi_1,\ldots,\varphi_n \in \Phi_{\mathcal A}$ and $n \in \mathbb N$. The smallest such $C$ is the \emph{BSE- norm} of $\sigma$ and will be denoted by $\|\sigma\|_{\text{BSE}}$. Let $C_{\text{BSE}}(\Phi_{\mathcal A})$ be the collection of all BSE- functions. Then $(C_{\text{BSE}}(\Phi_{\mathcal A}),\|\cdot\|_{\text{BSE}})$ is a commutative Banach algebra with pointwise operations \cite[Lemma 1]{T}. A commutative Banach algebra $\mathcal A$, which is without order, is a \emph{BSE- algebra} \cite{T} if $C_{\text{BSE}}(\Phi_{\mathcal A})=\widehat{M(\mathcal A)}$. The concept of BSE- algebra was introduced and studied by Takahasi and Hatori in  \cite{T} and subsequently for various Banach algebras \cite{P, I2, I, K2, Y}. The abbreviation BSE stands for Bochner-Schoenberg-Eberlein and refers to the classical theorem obtained by Bochner \cite{B} on the real line and an integral analogue by Schohenber \cite{S} and by  Eberlein \cite{E2} for a general locally compact abelian group $G$, stating that the group algebra $L^1(G)$ is a BSE- algebra (see also \cite[$\S$ 1.9.1]{T}).

Let $\mathcal A$ be a commutative Banach algebra, and let it be without order. Let $\text{span}(\Phi_{\mathcal A})$ be the linear span of $\Phi_{\mathcal A}$ in $\mathcal A^\ast$, and let $\mathcal K(\Phi_{\mathcal A})$ be the collection of all compact subsets of $\Phi_{\mathcal A}$. If $p \in \text{span}(\Phi_{\mathcal A})$, then $p$ can be uniquely represented as $p=\sum_{\varphi \in \Phi_{\mathcal A}}\widehat p(\varphi)\varphi$, where $\widehat p$ is a complex valued functions on $\Phi_{\mathcal A}$ with finite support. For a function $\sigma \in C_{BSE}(\Phi_{\mathcal A})$, define
$$\|\sigma \|_{BSE, \infty} =\inf_{K \in \mathcal K(\Phi_{\mathcal A})}\|\sigma\|_{BSE,K},$$ where, $$\|\sigma\|_{BSE,K}=\sup_{p\in \text{span}(\Phi_{A})}\left\lbrace \left|\sum_{\varphi \in \Phi_{A}} \widehat{p}(\varphi)\sigma(\varphi) \right|:\|p\|_{\mathcal A^{\ast}} \leq 1, \, \widehat{p}(\varphi) = 0\; (\varphi \in K) \right\rbrace.$$ Then by \cite[Corollary 3.9]{I} $C_{BSE}^{0}(\Phi_{\mathcal A}) =\{ \sigma \in C_{BSE}(\Phi_{\mathcal A}) : \| \sigma \|_{BSE, \infty} = 0 \}$ is a closed ideal in $(C_{BSE}(\Phi_{\mathcal A}),\|\cdot\|_{BSE})$. This ideal is referred as the \emph{Inoue-Doss ideal} associated with $\mathcal{A}$ \cite{Z}.  A commutative Banach algebra $\mathcal A$ is a \emph{BED- algebra} (Bochner - Eberlein - Doss- algebra) if $\widehat{\mathcal A} = C_{BSE}^{0}(\Phi_{A})$. This concept was introduced and studied by Takahasi and Inoue in \cite{I}. Doss in \cite[Theorem 2]{R2} gave the characterization of an absolutely continuous measures on $G$ with respect to the Haar measure. Later, Takahasi and Inoue generalized the same concept for any commutative Banach algebra in \cite{I} by defining BED- algebras.    

As said above, if $G$ is a locally compact abelian group, then $L^1(G)$ is both BSE- algebra and BED- algebra. The main results of the paper is as follows.

\begin{theorem}\label{bse}
If $\omega$ is a weight on a locally compact abelian group $G$ such that $\omega^{-1}$ is vanishing at infinity, then $L^1(G,\omega)$ is a BSE- algebra.
\end{theorem}
\begin{theorem}\label{bed}
If $\omega$ is a weight on a locally compact abelian group $G$ such that $\omega^{-1}$ is vanishing at infinity, then $L^1(G,\omega)$ is a BED- algebra.
\end{theorem}

\begin{remark}
If $\omega$ is a weight on a locally compact abelian group $G$, then by, \cite[Theorem 3.7.5]{H}, there is a continuous weight $\omega'$ on $G$ such that $L^1(G,\omega)$ and $L^1(G,\omega')$ are isomorphic as Banach algebras and hence $L^1(G,\omega)$ is a BSE (BED)- algebra if and only if $L^1(G,\omega')$ is a BSE (BED)- algebra. Thus to prove Theorem \ref{bse} and Theorem \ref{bed} we may assume that $\omega$ is continuous.
\end{remark}
So, all the weights in the paper are assumed to be continuous.
\section{Proofs}
First we note the following result.
\begin{theorem}\cite[$\S$ 1.2.6 Theorem(b)]{R}\label{1}
 If $K$ is a compact subset of $G$, $r>0$ and $U_{r} = \{z\in \mathbb C : |1-z|<r\}$, then $N(K,r) = \{\gamma \in \widehat G: \gamma(s) \in U_{r} \text{ for all } s \in K\}$ is open in $\widehat G$.
\end{theorem}

\begin{theorem}\label{thm1}
Let $\omega$ be a weight on a locally compact abelian group $G$ such that $\omega^{-1}$ is vanishing at infinity. Define $\Psi : \widehat G \rightarrow {L^{1}(G,\omega)}^{\ast}$ by $\Psi(\gamma) = \varphi_{\gamma}$ for all $\gamma \in \widehat G$. Then $\Psi$ is uniformly continuous on $\widehat G$.
\end{theorem}
\begin{proof}
Let $\epsilon > 0$. Then, by assumption, there exists a compact set $K$ of $G$ such that $\frac{1}{\omega(s)} < \frac{\epsilon}{4}$ for all $s \in K^{c}$. Now, let $U = N(K,\frac{\epsilon}{4})$. Then, by Theorem \ref{1}, $U$ is open in $\widehat G$ and the identity function $\mathbf 1_{\widehat G}$ is in $U$, where $\mathbf 1_{\widehat G}(s)=1$ for all $s \in G$. Let $f \in L^{1}(G,\omega)$ with $\|f\|_\omega\leq 1$, and let $\gamma_{1}, \gamma_{2} \in U$. Then
\begin{eqnarray*}
| \varphi_{\gamma_{1}}(f) - \varphi_{\gamma_{2}}(f)| & = &\left|  \int_{G} f(s)\overline{\gamma_{1}(s)}ds - \int_{G} f(s)\overline{\gamma_{2}(s)}ds \right |\\
& = & \left | \int_{G} f(s)(\overline{\gamma_{1}(s)} - \overline{\gamma_{2}(s)})ds \right|\\
& = & \left |\int_{K} f(s)(\overline{\gamma_{1}(s)} - \overline{\gamma_{2}(s)})ds + \int_{K^{c}} f(s)(\overline{\gamma_{1}(s)} - \overline{\gamma_{2}(s)}) ds \right|\\
& \leq &  \int_{K} |f(s)| |\gamma_{1}(s) - \gamma_{2}(s)|ds + \int_{K^{c}} |f(s)||\gamma_{1}(s) - \gamma_{2}(s)| ds \\
& \leq & \int_{K} |f(s)| |\gamma_{1}(s) - 1| + \int_{K} |f(s)| |1 - \gamma_{2}(s)|ds\\
&& + \int_{K^{c}} |f(s)| (|\gamma_{1}(s)| + |\gamma_{2}(s)|) \frac{\omega(s)}{\omega(s)} ds \\
& \leq & \int_{K} |f(s)| \frac{\epsilon}{4} ds + \int_{K} |f(s)| \frac{\epsilon}{4} ds + \int_{K^{c}} |f(s)|2 \frac{\omega(s)}{\omega(s)} ds \\
& \leq & \frac{\epsilon}{2} \int_{K} |f(s)|  ds +  2 \frac{\epsilon}{4} \int_{K^{c}} |f(s)|\omega(s)ds \\
& \leq & \frac{\epsilon}{2} \int_{K} |f(s)|\omega(s)  ds +  \frac{\epsilon}{2}\int_{K^{c}} |f(s)|\omega(s)ds \\
& = & \frac{\epsilon}{2} \int_{G} |f(s)|\omega(s)  ds=\frac{\epsilon}{2}\|f\|_\omega\\
& < & \epsilon.
\end{eqnarray*}
Since $f \in L^1(G,\omega)$ is arbitrary satisfying $\|f\|_\omega \leq 1$, we have $\|\varphi_{\gamma_{1}} - \varphi_{\gamma_{1}} \|_{L^{1}(G,\omega)^\ast} < \epsilon$. Thus we proved that $\|\Psi(\gamma_{1}) - \Psi(\gamma_{2}) \|_{L^{1}(G,\omega)^\ast} < \epsilon$ whenever $\gamma_1, \gamma_2 \in U$. Hence the proof.
\end{proof}

\begin{lemma} \label{lemm1}
Let $G$ be a locally compact abelian group, and let $\omega$ be a continuous weight on $G$ such that $\omega^{-1}$ is vanishing at infinity.	Let $\sigma\in C_{\text{BSE}}(\widehat{G}_\omega)$, and let $g:\widehat G \to \mathbb C$ be in $L^1(\widehat G)$. Then $$\left|\int_{\widehat G}g(\gamma)\sigma(\varphi_\gamma)d\gamma\right|\leq \|\sigma\|_{\text{BSE}}\left\|\int_{\widehat G}g(\gamma)\varphi_\gamma d\gamma\right\|_{L^1(G,\omega)^\ast}.$$
\end{lemma}
\begin{proof}
First we prove the result when $g \in C_c(\widehat G)$. Since the function $\Psi$ is uniformly continuous by Theorem \ref{thm1} and bounded and $g \in C_c(\widehat G)$, the integral $\int_{\widehat G}g(\gamma)\Psi(\gamma) d\gamma$ exists, i.e., $\int_{\widehat G}g(\gamma)\varphi_\gamma d\gamma$ is an element of $L^1(G,\omega)^\ast$. Let $K$ be the support of $g$. Let $(E_{i})_{i = 1}^{n}$ be a partition of $K$, and let $\gamma_{i} \in E_{i}$ for all $i=1,2,\ldots, n$. Then
\begin{equation*}
\left|\sum_{i=1}^n g(\gamma_i)\sigma(\varphi_{\gamma_i})m(E_{i})\right|  \leq  \|\sigma\|_{\text{BSE}}\left\|\sum_{i=1}^n g(\gamma_i)\varphi_{\gamma_i}m(E_{i})\right\|_{L^1(G,\omega)^\ast}.
\end{equation*}
Applying $\|P\|=\max\{m(E_{i}):1\leq i\leq n\}\to 0$, we have $$\left|\int_{K}g(\gamma)\sigma(\varphi_\gamma)d\gamma \right|\leq \|\sigma\|_{\text{BSE}}\left\|\int_{K}g(\gamma)\varphi_\gamma d\gamma \right\|_{L^1(G,\omega)^\ast}$$ or $$\left|\int_{\widehat G}g(\gamma)\sigma(\varphi_\gamma)d\gamma \right|\leq \|\sigma\|_{\text{BSE}}\left\|\int_{\widehat G}g(\gamma)\varphi_\gamma d\gamma \right\|_{L^1(G,\omega)^\ast}.$$ Let $M$ be a bound of $\sigma$, i.e., $|\sigma(\varphi_\gamma)|\leq M$ for all $\gamma \in \widehat G_\omega$. Now, let $g \in L^1(\widehat G)$ and take any $\epsilon>0$. Then there is $f \in C_c(\widehat G)$ such that $\|g-f\|_1=\int_{\widehat G}|f(\gamma)-g(\gamma)|d\gamma<\epsilon$. This gives $$\int_{\widehat G}|f(\gamma)-g(\gamma)||\sigma(\varphi_\gamma)|d\gamma\leq M\int_{\widehat G}|f(\gamma)-g(\gamma)|d\gamma <M\epsilon$$ and hence $$\left|\int_{\widehat G}g(\gamma)\sigma(\varphi_\gamma)dt\right|<M\epsilon+\left|\int_{\widehat G}f(\gamma)\sigma(\varphi_\gamma)dt\right|.$$ Now,
\begin{eqnarray*}
\left\|\int_{\widehat G}g(\gamma)\varphi_\gamma d\gamma-\int_{\widehat G}f(\gamma)\varphi_\gamma d\gamma \right\|_{L^1(G,\omega)^\ast} &\leq & \int_{\widehat G}|f(\gamma)-g(\gamma)|\|\varphi_\gamma\|_{L^1(G,\omega)^\ast}d\gamma\\
&\leq & \int_{\widehat G}|f(\gamma)-g(\gamma)|d\gamma<\epsilon \ (\text{as} \ \|\varphi_\gamma\|_{L^1(G,\omega)^\ast} \leq 1)
\end{eqnarray*}
and hence $$\left\|\int_{\widehat G}f(\gamma)\varphi_\gamma d\gamma \right\|_{L^1(G,\omega)^\ast}<\epsilon+\left\|\int_{\widehat G}g(\gamma)\varphi_\gamma d\gamma \right\|_{L^1(G,\omega)^\ast}.$$
So,
\begin{eqnarray*}
\left|\int_{\widehat G}g(\gamma)\sigma(\varphi_\gamma)d\gamma \right| & <& M\epsilon+\left|\int_{\widehat G}f(\gamma)\sigma(\varphi_\gamma)d\gamma\right|\\
& \leq & M\epsilon +\|\sigma\|_{\text{BSE}}\left\|\int_{\widehat G}f(\gamma)\varphi_\gamma d\gamma \right\|_{L^1(G,\omega)^\ast}\\
&< & M\epsilon+\|\sigma\|_{\text{BSE}}\,\epsilon+\|\sigma\|_{\text{BSE}}\left\|\int_{\widehat G}g(\gamma)\varphi_\gamma d\gamma \right\|_{L^1(G,\omega)^\ast}.
\end{eqnarray*}
Since $\epsilon>0$ is arbitrary, we have $\displaystyle{\left|\int_{\widehat G}g(\gamma)\sigma(\varphi_\gamma)d\gamma\right|\leq \|\sigma\|_{\text{BSE}}\left\|\int_{\widehat G}g(\gamma)\varphi_\gamma d\gamma \right\|_{L^1(G,\omega)^\ast}}$.
\end{proof}

The following lemma may be known but for completeness we are giving its proof.
\begin{lemma}\label{l}
Let $G$ be a locally compact abelian group, and let $$\mathscr{F} = \left\{f \in  C_{c}(G) :\widehat f \in L^1(\widehat G),\, f(s) = \int_{\widehat G}\widehat{f}(\gamma)\gamma(s)d\gamma\;(s \in G)\right\}.$$  Then $ \mathscr{F}$ is a dense subspace in $C_{0}(G)$.
\end{lemma}
\begin{proof}
Observe that $\mathscr{F}$ is a subspace of $C_{0}(G)$. Let $V$ be a neighborhood of $0$ in $G$. Choose a compact set $K$ such that $K - K \subset V$. Let $f$ be the characteristic function of $K$, divided by $\sqrt{m(K)}$. Define $g = f\ast \widetilde{f}$, where $\widetilde f(s)=\overline{f(-s)}$ for all $s \in G$. Then $g$ is continuous,  positive definite and $\widehat{g} \in L^1(\widehat G)$ \cite[Example 1.4.2]{R}. Therefore the inversion theorem applies to $g$, i.e., $g(s)=\int_{\widehat G}\widehat g(\gamma)\gamma(s)d\gamma$ for all $s \in G$. Also, $g(0) = \int_{\widehat G}\widehat{g}(\gamma)d\gamma = 1$ and $g$ is $0$ outside $K - K \subset V$. Thus $g \in \mathscr F$. Let $s \in G$ be such that $s \neq 0$. Choose a neighborhood $V$ of $0$ such that $s \notin V$. Then, by above argument, we get $g \in \mathscr F$ such that $g(0) = 1$ and $g(s)=0$. Hence $\mathscr{F}$ separates points of $G$. Therefore, by Stone-Weierstrass theorem, $\mathscr{F}$ is dense in $C_{0}(G)$.
\end{proof}

\begin{lemma}
Let $G$ be a locally compact abelian group, $\omega$ be a continuous weight on $G$, and let $\mathscr F$ be as in Lemma \ref{l}. Then $\mathscr{F} $ is a dense subspace in the Banach space $\left(C_{0}(G,\frac{1}{\omega}),\| \cdot \|_{0,{\omega}^{-1}}\right)$.
\end{lemma}
\begin{proof}
Let $f \in C_{0}(G,\frac{1}{\omega})$ and let $\epsilon > 0$. Then $\frac{f}{\omega} \in C_{0}(G)$. As $C_c(G)$ is dense in $C_0(G)$, there exists $g \in C_{c}(G)$ such that $\|\frac{f}{\omega} - g\|_{\infty} < \frac{\epsilon}{2}$. Since $g \in C_{c}(G)$, $\omega g  \in C_{c}(G)$. Now, $\| f - \omega g\|_{0,{\omega}^{-1}} = \|\frac{f}{\omega} - g\|_{\infty} < \frac{\epsilon}{2}$. As $\omega g  \in C_{c}(G) \subset C_{0}(G)$, by above lemma, there exists $h \in \mathscr{F}$ such that $\|h-\omega g\|_{\infty} < \frac{\epsilon}{2}$. Therefore, we have
\begin{eqnarray*}
\|f - h\|_{0,{\omega}^{-1}} &=&\|f-\omega g + \omega g - h\|_{0,{\omega}^{-1}}\\
&\leq & \|f-\omega g\|_{0,{\omega}^{-1}}+\|\omega g-h\|_{0,{\omega}^{-1}}\\
&\leq& \|f-\omega g\|_{0,{\omega}^{-1}}+\|\omega g - h\|_{\infty}\\
&<& \frac{\epsilon}{2} + \frac{\epsilon}{2}= \epsilon.
\end{eqnarray*}
Hence the proof.
\end{proof}
Now, we give a proof of Theorem \ref{bse}.
\begin{proof}[Proof of Theorem \ref{bse}]
By \cite[lemma 2.1]{G}, $L^1(G,\omega)$ has a bounded approximate identity.   Therefore, by \cite[corollary 5]{T}, $\widehat M(L^1(G,\omega))\subset C_{\text{BSE}}(\Phi_{L^1(G,\omega)})$. Let $\sigma$ be in $C_{\text{BSE}}(\Phi_{L^1(G,\omega)})$. Then
\begin{eqnarray}\label{BSE}
\left|\sum_{i=1}^n c_i\sigma(\varphi_{i})\right|\leq \|\sigma\|_{\text{BSE}}\left\|\sum_{i=1}^nc_i\varphi_{i}\right\|_{L^1(G,\omega)^\ast}
\end{eqnarray}
for all $c_1,c_2,\ldots,c_n \in \mathbb C$, $\varphi_1,\varphi_2,\ldots,\varphi_n \in \Phi_{L^1(G,\omega)}$ and $n \in \mathbb N$. If $\psi \in L^1(G)^\ast \cap L^1(G,\omega)^\ast$, then $\|\psi\|_{L^1(G,\omega)^\ast}=\sup\{|\psi(f)|:f \in L^1(G,\omega), \|f\|_\omega \leq 1\}\leq \sup\{|\psi(f)|:f \in L^1(G), \|f\|_1\leq 1\}=\|\psi\|_{L^1(G)^\ast}$ as if $f \in L^1(G,\omega)$, then $\|f\|_1\leq \|f\|_\omega$. Let $\gamma_1,\gamma_2,\ldots,\gamma_n \in \widehat G$ and $c_1,c_2,\ldots,c_n \in \mathbb C$. Then $\sum_{i=1}^nc_i\varphi_{\gamma_i} \in L^1(G)^\ast \cap L^1(G,\omega)^\ast$. So, by inequality (\ref{BSE}), we have
\begin{eqnarray*}
\left|\sum_{i=1}^n c_i\sigma(\varphi_{\gamma_i})\right|\leq \|\sigma\|_{\text{BSE}}\left\|\sum_{i=1}^nc_i\varphi_{\gamma_i}\right\|_{L^1(G,\omega)^\ast}\leq \|\sigma\|_{\text{BSE}}\left\|\sum_{i=1}^nc_i\varphi_{\gamma_i}\right\|_{L^1(G)^\ast}.
\end{eqnarray*}
It means that $\sigma$ is a BSE- function  on $\Phi_{L^1(G)}$. Since $L^1(G)$ is a BSE- algebra \cite{E}  and the multiplier algebra $M(L^1(G))$ of $L^1(G)$ is the Banach algebra $M(G)$ of all complex regular Borel measures on $G$ \cite[Thorem 0.1.1]{L}, there is $\mu \in M(G)$ such that
\begin{equation*}
\sigma(\varphi_\gamma) =  \widehat \mu(\varphi_\gamma) = \int_{G}\overline{\gamma(s)}d\mu(s)\quad(\gamma \in \widehat{G}).
\end{equation*}
We prove that $\overline{\mu}\in M(G,\omega)$, that is  $\int_{G}\omega(s)d|\overline\mu|(s)<\infty$ and to do this we prove that the mapping $g\mapsto \int_{G}g(s)d\overline\mu(s)$ from $C_0(G,\frac{1}{\omega})$ to $\mathbb C$ is a bounded linear map. Since the dual of $C_0(G,\frac{1}{\omega})$ is $M(G,\omega)$, this will complete the proof.
	
Let $\mathscr F$ be as in Lemma \ref{l}. Define $\eta$ on $\mathscr F$ by
\begin{eqnarray*}
\eta( f)=\int_{G} f(s)d\overline\mu(s)\quad( f \in \mathscr F).
\end{eqnarray*}
The above integral makes sense as $f$ is bounded and $\overline\mu$ is a complex measure. Observe that $\eta$ is linear on $\mathscr F$. Let $f \in \mathscr F$. Then
\begin{eqnarray*}
\eta( f) & = & \int_{G} f(s)d\overline\mu(s) = \int_{G}\left(\int_{\widehat G} \widehat f(\gamma)\gamma(s)d\gamma\right)d\overline\mu(s)\\
& = & \int_{\widehat G}\widehat f(\gamma)\left(\int_{G}\gamma(s)d\overline\mu(s)\right)d\gamma\\
& = & \int_{\widehat G}\widehat f(\gamma)\overline{\left(\int_{G}\overline{\gamma(s)}d\mu(s)\right)}d\gamma\\
& = & \int_{\widehat G}\widehat f(\gamma)\overline{\widehat\mu(\varphi_\gamma)}d\gamma\\
& = & \int_{\widehat G}\widehat f(\gamma)\overline{\sigma(\varphi_\gamma)}d\gamma.
\end{eqnarray*}
Since $\sigma \in C_{\text{BSE}}(\Phi_{L^1(G,\omega)})$ and $\widehat f \in  L^1(\widehat G)$, we have, by lemma 1.3 ,
\begin{eqnarray}\label{e}
|\eta( f)| & =& \left|\int_{\widehat G}\widehat f(\gamma)\overline{\sigma(\varphi_\gamma)}d\gamma\right|=\left|\int_{\widehat G}\overline{\widehat f(\gamma)}\sigma(\varphi_\gamma)d\gamma\right| \nonumber\\
&\leq & \|\sigma\|_{\text{BSE}}\left\|\int_{\widehat G}\overline{\widehat f(\gamma)} \varphi_\gamma d\gamma\right\|_{L^1(G,\omega)^\ast}.
\end{eqnarray}
Let $g \in L^1(G,\omega)$. Then
\begin{eqnarray*}
\left|\left(\int_{\widehat G}\overline{\widehat f(\gamma)}\varphi_\gamma dt\right)(g)\right|& = & \left|\int_{\widehat G}\overline{\widehat f(\gamma)}\varphi_\gamma(g)d\gamma\right|=\left|\int_{\widehat G}\overline{\widehat f(\gamma)}\widehat g(\varphi_\gamma)d\gamma\right|\\
& = & \left|\int_{\widehat G}\overline{\widehat f(\gamma)}\left(\int_{G}g(s)\overline{\gamma(s)}ds\right)d\gamma\right|\\
& = & \left|\int_{G}g(s)\left(\int_{\widehat G}\overline{\widehat f(\gamma)}\overline{\gamma(s)}d\gamma\right)ds\right|\\
& = & \left|\int_{G}g(s) \overline{f(s)}ds\right|\\
& \leq & \int_{G}|g(s)|\omega(s)\frac{|\overline{f(s)}|}{\omega(s)}ds\\
& \leq & \| f\|_{0,\omega^{-1}}\|g\|_\omega.
\end{eqnarray*}
This proves that $\displaystyle{\left\|\int_{\widehat G}\overline{\widehat f(\gamma)}\varphi_\gamma d\gamma\right\|_{L^1(G,\omega)^\ast}\leq \| f\|_{0,\omega^{-1}}}$ and hence, by equation (\ref{e}), $|\eta( f)|\leq \|\sigma\|_{\text{BSE}}\| f\|_{0,\omega^{-1}}$ for all  $ f \in \mathscr F$. As $\mathscr F$ is dense in $(C_0(G,\omega^{-1}),\|\cdot\|_{0,\omega^{-1}})$, $\eta$ has a unique norm preserving extension on $C_0(G,\omega^{-1})$ and this extension will be $$\eta(f)=\int_G f(s)d\overline\mu(s)\quad(f \in C_0(G,\omega^{-1})).$$
This implies that $\overline\mu \in M(G,\omega)$ and hence $\mu\in M(G,\omega)$. This completes the proof.
\end{proof}

We first note the following results. These results will be useful to us in proving the Theorem \ref{bed}.
\begin{theorem}\cite[Theorem 1]{R2}\label{TR1}
Let $G$ be a locally compact abelian group. Then a complex valued continuous function $\sigma$ on $\widehat{G}$ is the Fourier-Stieltjes transform of a singular measure $\mu_s$ on $G$ if and only if
\begin{enumerate}
\item $|\sum_{\gamma \in \widehat G} \widehat p(\varphi_\gamma) \sigma(\varphi_{\gamma})| \leq \|\mu_s\|$ whenever $p = \sum_{\gamma \in \widehat G}\widehat p(\varphi_\gamma)\varphi_\gamma \in \text{span}(\Phi_{L^1(G)})$ and $\|p\|_{L^1(G)^\ast} \leq 1$;
\item given $\epsilon > 0$ and the compact set $K$ in $\widehat{G}$ there is $p = \sum_{\gamma \in \widehat G\setminus K} \widehat p(\varphi_\gamma)\varphi_\gamma \in \text{span}(\Phi_{L^1(G)})$ such that $\|p\|_{L^1(G)^\ast} \leq 1$ and $|\sum_{\gamma \in \widehat G\setminus K} \widehat p(\varphi_\gamma) \sigma(\varphi_\gamma)| > \|\mu_s\| - \epsilon$.
\end{enumerate}
\end{theorem}

\begin{theorem}\cite[Theorem 2]{R2}\label{TR2}
Let $G$ be a locally compact abelian group. Then a complex valued continuous function $\sigma$ on $\widehat{G}$ is the Fourier-Stieltjes transform of an absolutely continuous measure $\mu$ on $G$ if and only if
\begin{enumerate}
\item $|\sum_{\gamma \in \widehat G} \widehat p(\varphi_\gamma)\sigma(\varphi_\gamma)| \leq \|\mu\|$ whenever $p = \sum_{\gamma \in \widehat G} \widehat p(\varphi_\gamma)\varphi_\gamma \in \text{span}(\Phi_{L^1(G)})$ and $\|p\|_{L^1(G)^\ast} \leq 1$;
\item given $\epsilon > 0$ there is a compact set $K$ in $\widehat G$ such that $|\sum_{\gamma \in \widehat G\setminus K} \widehat p(\varphi_\gamma) \sigma(\varphi_\gamma)| < \epsilon$ whenever $p = \sum_{\gamma \in \widehat G\setminus K} \widehat p(\varphi_\gamma)\varphi_\gamma\in \text{span}(\Phi_{L^1(G)})$ and $\|p\|_{L^1(G)^\ast} \leq 1$.
\end{enumerate}
\end{theorem}

\begin{proof}[Proof of Theorem \ref{bed}] We consider the following two cases.\\
Case 1: Let $G$ be a discrete group. Since $L^1(G)$ and $L^1(G,\omega)$ have identities, $M(L^1(G))=L^1(G)$ and $M(L^1(G,\omega))=L^1(G,\omega)$. So, $ C_{BSE}(\widehat{G}) = L^{1}(G)$. Now, let $\sigma \in C_{BSE}^{0}(\widehat G_\omega) $. Then $\sigma \in C_{BSE}(\widehat G_ \omega) = \widehat{L^{1}(G, \omega)}$, by Theorem \ref{bse}, hence we have $\sigma=\widehat{f}$ for some $f \in  L^{1}(G,\omega)$. Therefore $L^{1}(G,\omega)$ is a BED in this case.\\

Case 2: Let $G$ not be  discrete. Then $\widehat{G}$ is not compact. Let $\sigma \in C_{BSE}^{0}(\Phi_{L^1(G, \omega)}) $.  Then, by Theorem \ref{bse}, we have $\sigma = \widehat{\mu}$ for some  $\mu \in M(G,\omega)\subset M(G)$. Let $ d\mu = d\mu_{s} + fdm$ be the Lebesgue decomposition of $\mu$ with respect to the Haar measure $m$, where $\mu_{s}$ is a complex measure on $G$ which is mutually singular to $m$ and $f \in L^{1}(G)$. We show that $\|\mu_s\|=0$ and hence $\mu_s=0$.  Let $\epsilon > 0$ be given. As $\sigma \in C_{BSE}^{0}(\Phi_{L^1(G, \omega)})$, by \cite[Definition 3.5]{I}, there exists a compact set $K'$ in $\widehat G_\omega$ such that if $p= \sum_{\gamma \in \widehat G_\omega\setminus K'} \widehat p(\varphi_\gamma)\varphi_\gamma \in \text{span}(\Phi_{L^1(G,\omega)})$ and $\|p\|_{L^1(G,\omega)^\ast} \leq 1$, then $|\sum_{\gamma \in \widehat G_\omega\setminus K'} \widehat p(\varphi_\gamma)\sigma(\varphi_\gamma)| \leq \epsilon$. Also, for  the same $\epsilon > 0$, by Theorem \ref{TR2}(2), we get a compact subset $K_{1}$ of $\widehat{G}$ such that $|\sum_{\gamma \in \widehat G\setminus K_1} \widehat p(\varphi_\gamma) \widehat f(\varphi_\gamma)| < \epsilon$ whenever $p = \sum_{\gamma \in \widehat G\setminus K_1} \widehat p(\varphi_\gamma)\varphi_\gamma\in \text{span}(\Phi_{L^1(G)})$ and $\|p\|_{L^1(G)^\ast} \leq 1$. Let $K = K_{1} \cup (K' \cap \widehat{G})$. Since $\widehat{G}$ is a closed subset of $\widehat G_\omega$, $K$ is a compact subset of $\widehat G$. By Theorem \ref{TR1}(2),  we get $q=\sum_{\gamma \in \widehat G\setminus K}\widehat q(\varphi_\gamma)\varphi_\gamma \in \text{span}(\Phi_{L^1(G)})$ with  $\|q\|_{L^1(G)^\ast} \leq 1$ and $|\sum_{\gamma \in \widehat G\setminus K} \widehat q(\varphi_\gamma) \widehat{\mu_s}(\varphi_\gamma)| > \|\mu_s\| - \epsilon$. Also,  $|\sum_{\gamma \in \widehat G\setminus K} \widehat q(\varphi_\gamma) \widehat{f}(\varphi_\gamma)| < \epsilon$ as $K_1 \subset K$. Also, observe that $|\sum_{\gamma \in \widehat G_\omega\setminus K}\widehat q(\varphi_\gamma)\sigma(\varphi_\gamma)|<\epsilon$ as $ (K' \cap \widehat{G}) \subset K$ and  $\|q\|_{L^1(G)^\ast} \leq 1$ implies $\|q\|_{L^1(G,\omega)^\ast} \leq 1$.  So,
\begin{eqnarray*}
\|\mu_s\| - \epsilon < \left|\sum_{\gamma \in \widehat G\setminus K} \widehat q(\varphi_\gamma) \widehat{\mu_s}(\varphi_\gamma)\right|
= \left|\sum_{\gamma \in \widehat G\setminus K} \widehat q(\varphi_\gamma) \sigma(\varphi_\gamma) - \sum_{\gamma \in \widehat G\setminus K} \widehat q(\varphi_\gamma) \widehat f(\varphi_\gamma)\right| <  2\epsilon
\end{eqnarray*}
Since, $\epsilon > 0$ is arbitrary, we have $\|\mu_s\|=0$, i.e., $\mu_{s} = 0$. Hence we have $\widehat{\mu} = \widehat{f}$ or $d\mu=fdm$. Since $\mu \in M(G,\omega)$, it follows that $f \in L^{1}(G, \omega)$ and hence $\sigma = \widehat{f}$. This proves that $L^1(G,\omega)$ is a BED- algebra.
\end{proof}

\begin{remark}
We do not know whether Theorem \ref{bse} and Theorem \ref{bed} hold true without the assumption that $\omega^{-1}$ is vanishing at infinity.
\end{remark}

\end{document}